\documentclass[12pt, reqno]{amsart}
\usepackage{amsmath, amsthm, amscd, amsfonts, amssymb, graphicx, color}

\newtheorem{df}{Definition}[section]
\newtheorem{thm}[df]{Theorem}
\newtheorem{pro}[df]{Proposition}

\begin{document}
\setcounter{page}{1}

\title[Lie algebras of Operators]{Tensor Products and Joint Spectra \\  for solvable Lie Algebras of 
operators }
\author{Enrico Boasso}

\begin{abstract} Given two complex Hilbert spaces, $H_1$ and $H_2$, and two complex solvable finite dimensional Lie
algebras of operators, $L_1$ and $L_2$, such that $L_i$ acts on $H_i$ (i= 1,2), the
joint spectrum of the Lie algebra $L_1\times L_2$, which acts on $H_1\overline\otimes H_2$, is expressed by
the cartesian product of $Sp(L_1,H_1)$ and $Sp(L_2,H_2)$.\end{abstract}
\maketitle
\section{ Introduction}
\indent J. L. Taylor developed in [6] a notion of joint spectrum for an $n$-tuple
$a$, $a=(a_1,...,a_n)$, of mutually commuting operators acting on a Banach space
$E$, i.e., $a_i\in{\mathcal L} (E)$, the algebra of all bounded linear operators on
$E$, and $[a_i,a_j] =0$, $1\le i, j\le n$. This interesting notion, which
extends in a natural way the spectrum of a single operator, has many important
properties, among then, the projection property and the fact that $Sp(a,E)$ is
a compact non empty subset of $\Bbb C^n$, where $Sp(a,E)$ denotes the joint spectrum
of $a$ in $E$.\par
\indent One of the most remarkable results of the Taylor joint spectrum is the one 
related with tensor products of tuples of operators. For example, in [2], Z.
Ceausescu and F. H. Vasilescu proved the following result. 
Let $H_i$, $1\le i\le n$, be complex Hilbert spaces, and $a_i $, $1\le i\le n$,
be bounded linear operators defined on $H_i$, respectively, $1\le i\le n$. If we denote 
by $H$ the completion of the tensor product $H_1\otimes \ldots\otimes H_n$
with respect to the canonical scalar product, we may consider the $n$-tuple
of operators $\tilde a$, $\tilde a= (\tilde a_1,\ldots ,\tilde a_n)$, where
$\tilde a_i = 1\otimes\ldots\otimes 1\otimes a_i\otimes 1\otimes\ldots\otimes 1$, $1\le i\le n$, and 
1 denotes the identity of the corresponding spaces. Then the following identity holds,
$$
Sp(a,E)=Sp(a_1)\times\ldots\times Sp(a_n).
$$
\indent Furthermore, in [3], Z. Ceausescu and F. H. Vasilescu showed that if
$H_1$ (resp. $H_2$) is a complex Hilbert space and $a=(a_1,\dots ,a_n) $, (resp. $b=(b_1,\dots ,b_m)$), is a mutually commuting tuple of operators
acting on $H_1$, (resp. $H_2$), then, the commuting tuple $(\tilde a ,\tilde b )=
(a_1\otimes 1,\ldots ,a_n\otimes 1,1\otimes b_1,\ldots ,1\otimes b_m)$
in  ${\mathcal L} (H_1\overline\otimes H_2)$, satisfies the relation,
$$
Sp((\tilde a,\tilde b), H_1\overline\otimes H_2)= Sp(a,H_1)\times Sp(b, H_2),
$$
where 1 denotes the identity map of the corresponding Hilbert spaces, and $H_1\overline\otimes H_2$ is
the completion of the tensor product $H_1\otimes H_2$ with respect to the canonical scalar product, see [3,Theorem 2.2].\par
\indent In [1] we defined a joint spectum for  complex solvable finite dimensional Lie algebras of operators $L$, acting on a Banach space $E$, and
we denoted it by $Sp(L,E)$. We  proved that $Sp(L,E)$ is a compact non
empty subset of $L^* $ and that the projection property for ideals still holds. Besides, when  $L$
is a commutative algebra, our spectrum reduces to Taylor joint spectrum in the following sense. If $\dim L =n$,
$\{a_i\}_{(1\le i\le n)}$ is a basis of $L$ and we consider the $n$-tuple
$a=(a_1,\ldots ,a_n)$, then $\{ (f(a_1),\ldots ,f(a_n)): f\in Sp(L,E)\}= Sp(a,E)$,
i.e., $Sp(L,E)$ in terms of the basis of $L^*$ dual of $\{a_i\}_{(1\le i\le n)}$
coincides with the Taylor joint spectrum of the $n$-tuple $a$. 
Then, the following question arises naturally. If $H_i$, $i= 1, 2$, are two complex  Hilbert spaces,
and $L_i$, $i= 1, 2$, are two complex solvable finite dimensional Lie algebras of 
operators such that $L_i$ acts on $H_i$, respectively, $ i= 1, 2$, is there any relation between 
$Sp(L_1\times L_2, H_1\overline\otimes H_2)$ and $ Sp(L_1,H_1)\times S(L_2,H_2)$. \par
\indent In this paper we answer this question in the affirmative. Moreover, 
by a refinement of the argument of Z. Ceausescu and F. H. Vasilescu in
[3], we extend the main results of [2] and [3] for complex solvable finite dimensional 
Lie algebras and its joint spectrum. In order to describe in more detail our main theorem we need to introduce a definition. 
If $H_i$ and  $L_i$ are as above, $i= 1, 2$,      
we consider the  direct product of $L_1$ and $L_2$, i.e., the complex solvable finite dimensional Lie algebra $L_1\times L_2$ defined by,
$$
L_1\times L_2 =\{ x_1\otimes 1+ 1\otimes x_2: x_i \in L_i, i=1,2\},
$$
where 1 is as above. Then, its is clear that $L_1\times L_2$ is a Lie algebra of operators which acts on $H_1\overline\otimes H_2$, and our main theorem may be stated as follows, 

$$
Sp(L_1\times L_2, H_1\overline\otimes H_2) = Sp(L_1,H_1)\times Sp(L_2, H_2),
$$
where the above sets are considered as subsets of $(L_1\times L_2)^*$ under
the natural identification $(L_1\times L_2)^*\cong L_1^*\times L_2^*$.\par
\indent The paper is organized as follows. In Section 2 we review several definitions
and results of [1] and we also prove a proposition which is an important steps
to our main result. Finally, in Section 3, we prove our main theorem.\par
\section{Preliminaries}
\indent We briefly recall several definitions and results related to the spectrum
of a complex solvable Lie algebra of operators, see [1]. From now on $L$ denotes a complex solvable finite dimensional Lie
algebra and $H$  a complex Hilbert space on which $L$ acts as right continous operators,
i.e., $L$ is a Lie subalgebra of ${\mathcal L}(H)$ with the opposite product. 
If $\dim L =n$ and  $f$ is a character of $L$, i.e., $f\in L^*$ and $f(L^2) = 0$, where $L^2 = \{ [x,y]: x, y \in L\}$, then consider 
the following chain complex, $(H\otimes\wedge L, d(f))$, where $\wedge L$ denotes
the exterior algebra of $L$, and $d_p(f)$ is as follows,
$$
d_p (f)\colon H\otimes\wedge^p L\rightarrow H\otimes\wedge^{p-1} L,   
$$

\begin{align*} d_p (f) e\langle x_1\wedge\dots\wedge x_p\rangle & = \sum_{k=1}^{k=p}(-1)^{k+1}e(x_k-f(x_k))\langle x_1\wedge\ldots\wedge\hat{x_k}\wedge\dots\wedge x_p\rangle\\
                                                         + &\sum_{1\le k< l\le p} (-1)^{k+l}e\langle [x_k, x_l]\wedge x_1\wedge\ldots\wedge\hat{x_k}\wedge\ldots\wedge\hat{x_l}\wedge\ldots \wedge x_p\rangle ,\\ \end{align*} 

\noindent where $\hat{ }$ means deletion. If $p\le 0$ or $p\ge {n+1}$, we  define $d_p (f) =0$.\par
\indent  If we denote by $H_*((H\otimes\wedge L,d(f)))$ the homology of the complex
$(H\otimes\wedge L, d(f))$, we may state our first definition.\par
\begin{df} With $L$ and $f$ as above, the
set $\{f\in L^*: f(L^2) =0, H_*((H\otimes\wedge L,d(f)))\neq 0\}$ is the joint spectrum
of $L$ acting on $H$, and it is denoted by $Sp(L,H)$.\end{df}
\indent As we have said, in [1] we proved that $Sp(L,E)$ is a compact non empty subset of $L^*$ which 
reduces to Taylor joint spectrum when $L$ is a commutative algebra, in the sense 
explained in the introduction. Besides,
if $I$ is an ideal of $L$ and $\pi$ denotes the projection map from $L^*$ to $I^*$, then,
$$
Sp(I,H) = \pi (Sp(L,H)),
$$
i.e., the projection property for ideals still holds.
 With regard to this property, we ought to mention the paper of C. Ott, see [5],
who pointed out a gap in the proof of this result, and give another proof of it.
In any case, the projection property remains true.\par
\indent We observe that  the set $H\otimes\wedge L$ has a natural structure
of Hilbert space, so that the sets $H\otimes \langle x_{i_1}\wedge\dots \wedge x_{i_p}\rangle$, 
$1\le i_1<\ldots <i_p\le n$, $0\le p\le n$,
are orthogonal subspaces of $H\otimes\wedge L$, and if $<,>$ denotes
the inner product of $H$, $<a\langle x_{i_1}\wedge\ldots \wedge x_{i_p}\rangle, b\langle  x_{i_1}\wedge\ldots\wedge x_{i_p}\rangle> =<a,b>$.\par
\indent We shall have occasion to use the direct product of two complex solvable 
finite dimensional Lie algebras and its action on the tensor product of two 
complex Hilbert spaces. We recall here the main facts which we need for our work.  
If $H_i$, i =1,2, are two complex Hilbert spaces, then $H_1\overline\otimes H_2$ 
denotes the completion of the tensor product $H_1\otimes H_2$ with respect to 
the canonical scalar product. Now, if $L_i$, i=1,2, are two complex solvable finite 
dimensional Lie algebras of operators, such that $L_i$ acts on $H_i$, respectively, i=1,2, 
we consider the algebra  $L_1\times L_2$, the direct 
product of $L_1$ and $L_2$, which acts in a natural way on  $ H_1\overline\otimes H_2$, and it is defined by, 
$$
L_1\times L_2 =\{ x\otimes 1+ 1\otimes y: x\in L_1+ y\in L_2\},
$$
where $1$ denotes the identity of the corresponding space.\par
\indent It is clear that $L_1\times L_2$, defined as above, is a complex solvable
 finite dimensional Lie subalgebra of ${\mathcal L} (H_1\overline\otimes H_2)$.
Moreover, by the structure of the Lie bracket in $L_1\times L_2$, we have two distinguished ideals,
$L'_1$ and $L'_2$, which we define as follows,
$$
L^{'}_1 =\{x\otimes 1: x\in L_1\},
\hskip3cm
L_2^{'} = \{1\otimes y: y\in L_2\}.
$$
\indent In addition, if we consider the natural identification $\tilde{K}:(L_1\times L_2)^*\cong L_1^*\times L_2^*$,
$\tilde{K}(f)=(f\circ i_1,f\circ i_2)$, where $f\in (L_1\times L_2)^*$ and $i_j\colon L_j\to L_1\times L_2$, $j=1,2$,
are the canonical inclusions, as $(L_1\times L_2)^2=L_1^2\times L_2^2$, we have that the set of characters of
$L_1\times L_2$ is the cartesian product of the sets of characters of $L_1$ and $L_2$.\par
\indent The following proposition is an important step to our main theorem.
\begin{pro} Let $H_i$, $i= 1,2$, be two complex Hilbert spaces, and $L$ a complex  solvable 
finite dimensional Lie algebra of operators acting on $H_1$. Let $L{'}_1$, (resp. $L{'}_2$),
be the Lie algebra of operators $\{ x\otimes 1: x\in L\}$, (resp. $\{1\otimes x: x\in L\}$), which acts on $H_1\overline \otimes H_2$, (resp. $H_2\overline\otimes H_1$),
where 1 denotes the identity map of $H_2$, (resp. $H_1$). Then,\par
\noindent {\rm (i)} $ Sp(L_1^{'}, H_1\overline\otimes H_2) = Sp(L^{'}_2, H_2\overline\otimes H_1)$,\par
\noindent {\rm (ii)} $ Sp(L^{'}_1, H_1\overline\otimes H_2)\subseteq Sp(L,H_1)$,\par
\noindent {\rm (iii)} $ Sp(L^{'}_2, H_2\overline\otimes H_1)\subseteq Sp(L,H_1)$.\end{pro}
\begin{proof} It is clear that (iii) is a consequence of (i) and (ii). Let us prove (i).\par 
\indent If $f$ is a character of $L$, we consider $C_1$ (resp. $ C_2$), the Koszul
complex associated to the Lie algebra $L_1^{'}$ (resp. $L^{'}_2$) and $f$, then,
$$
C_1 = (H_1\overline\otimes H_2\otimes\wedge L_1^{'}, d^1(f)),
\hskip2cm
C_2 =(H_2\overline\otimes H_1\otimes \wedge L_2^{'}, d^2(f)),
$$
where the maps $d^i(f)$, i= 1,2, are as above.\par
\indent In addition, a elementary calculation shows that the map $\mu$,
$$
\mu_p\colon H_1\overline\otimes H_2\otimes\wedge^p L_1^{'}\to H_2\overline\otimes H_1\otimes\wedge^p L^{'}_2,
$$
$$
\mu_p(e_1\otimes e_2\langle x_1,\dots ,x_p\rangle) = e_2\otimes e_1\langle x_1, \ldots ,x_p\rangle,
$$
defines an isomorphism which commutes with $d^i(f)$, i=1,2, i.e., $\mu$ is
an isomorphism of chain complexes. Then,
$$
Sp(L_1^{'}, H_1\overline\otimes H_2) = Sp(L_2^{'}, H_2\overline\otimes H_1).
$$
\indent In order to verify (ii), let us consider the Koszul complex associated to $L$
and $f$, $C= (H_1\otimes\wedge ^p L, d(f))$, and $\tilde C $ the following chain complex,
$$
\tilde C = ((H_1\otimes\wedge L)\overline \otimes H_2, d(f)\otimes 1),
$$
where 1 denotes the identity map of $H_2$.\par
\indent We observe that if $\eta$ is the map defined by,
$$
\eta\colon (H_1\overline\otimes H_2)\otimes\wedge^p L_1^{'}\to (H_1\otimes\wedge^p L)\overline\otimes H_2,
$$
$$
\eta_p(e_1\otimes e_2\langle x_1,\dots ,x_p\rangle) = e_1\langle x_1, \ldots ,x_p\rangle\otimes e_2,
$$
then an easy  calculation shows that $\eta$ defines an isomorphism of chain complexes
between $C_1$ and $\tilde C$.\par
\indent Now,  if $f$ does not belong to $Sp(L, H_1)$, by [8, Lemma 2.4], the complex $\tilde C$
is exact. As $\eta$ is an isomorphism of chain complexes, if $C_1$ is exact, then $f$   
 does not belong
to $Sp( L_1^{'}, H_1\overline\otimes H_2)$.\par
\end{proof}

 \section{The Main Result}
\indent We now state our main result.\par
\begin{thm} Let $H_1$ and $H_2$  be two complex Hilbert spaces and $L_1$
and $L_2$  two complex solvable finite dimensional Lie algebras of operators
such that $L_i$ acts on $H_i$, respectively, i= 1,2. Let us consider the complex solvable finite dimensional
Lie algebra $L_1\times L_2$, which acts on $H_1\overline\otimes H_2$ and it is defined by,
$$
L_1\times L_2 =\{x_1\otimes 1+ 1\otimes x_2: x_i\in L_i,\hbox{  } i= 1,2\},
$$
where 1 denotes  the identity of the corresponding spaces.\par
Then,
$$
Sp(L_1\times L_2, H_1\overline\otimes H_2) = Sp(L_1, H_1)\times Sp(L_2, H_2),
$$
where, in the above equality, the set $Sp(L_1\times L_2,H_1\overline\otimes H_2)$
is considered as a subset of $L_1^*\times L_2^*$ under the natural identification
$\tilde{K}\colon (L_1\times L_2)^*\cong L_1^*\times L_2^*$ of Section 2. 
\end{thm}

\begin{proof}
\indent In order to prove that $Sp(L_1\times L_2,H_1\overline\otimes H_2)$ is contained 
in the cartesian product of $Sp(L_1,H_1)$ and $Sp(L_2,H_2)$, let $L_1^{'}$ (resp. $ L_2^{'}$) be the ideal of $L_1\times L_2$ defined
by $\{ x\otimes 1: x\in L_1\}$ (resp. $\{ 1\otimes y: y\in L_2\}$), where 1 is as above, see Section 2. Then, by the 
projection property of the joint spectrum, Proposition 1 and the identification $\tilde{K}$,

\begin{align*}
Sp(L_1\times L_2, H_1\overline\otimes H_2) & \subseteq Sp(L^{'}_1, H_1\overline\otimes H_2)\times Sp(L_2^{'}, H_1\overline\otimes H_2)\\
                                           & \subseteq Sp(L_1, H_1)\times Sp(L_2,H_2).\\ \end{align*}
\indent Let us prove the converse inclusion. We consider  the following preliminary facts first.\par
\indent Let $f_i$ be a character of $L_i$ and  $K_i$
 the Koszul complex of $L_i$ acting on $H_i$ associated to $f_i$, for  $i = 1,2 $,
$$
K_1 = (H_1\otimes \wedge L_1, d^1(f_1)), 
\hskip2cm
 K_2 = (H_2\otimes \wedge L_2, d^2(f_2)).
$$
\indent We observe that there is a natural identification between the spaces $ (H_1\otimes H_2)\otimes\wedge (L_1\times L_2)$
and $ (H_1\otimes \wedge L_1)\otimes (H_2\otimes\wedge L_2)$. If we denote by $\psi$
this identification, $\psi$ is the following map,
$$
\psi\colon (H_1\otimes H_2)\otimes \wedge (L_1\times L_2)\to (H_1\otimes \wedge L_1)\otimes (H_2\otimes \wedge L_2),
$$
$$
\psi(e_1\otimes e_2\langle x_1\wedge\dots \wedge x_p\wedge y_1\wedge\ldots \wedge y_q\rangle) = e_1\langle x_1\wedge\dots \wedge x_p\rangle\otimes e_2\langle y_1\wedge\dots \wedge y_q\rangle,
$$
where $ e_1\in H_1$, $e_2\in H_2$, $p\in [\![1, n]\!]$, $q \in [\![1, m]\!]$,  $n = \dim (L_1)$, 
and $m = \dim (L_2)$.\par
\indent As
$$
(H_1\overline \otimes H_2)\otimes\wedge ^k(L_1\times L_2) = \oplus_{p+q = k} (H_1\otimes\wedge^p L_1)\overline\otimes (H_2\otimes \wedge^q L_2),
$$
if we consider $H_i\otimes\wedge L_i$, i=1,2, $(H_1\overline\otimes H_2)\otimes\wedge (L_1\times L_2)$
and $(H_1\otimes\wedge L_1)\overline\otimes (H_2\otimes \wedge L_2)$ with their natural
structure of Hilbert spaces, a straightforward calculation shows that the 
map $\psi$ may be extended  to an isometric isomorphism from $ (H_1\overline\otimes H_2)\otimes\wedge (L_1\times L_2)$
onto  $(H_1\otimes\wedge L_1)\overline\otimes (H_2\otimes\wedge L_2)$.\par
\indent Besides, if $f$ is a character of $L_1\times L_2$ and if we  define $ f_j=f\circ i_j \in L_j^*$,
where $i_j\colon L_j\to L_1\times L_2$ are the canonical inclusions and $j= 1,2$, i.e., if we consider $f$ decomposed under the natural identification
$\tilde{K}\colon (L_1\times L_2)^*\cong L_1^*\times L_2^*$, if $K$ is the Koszul complex of $L_1\times L_2$
acting on $H_1\overline\otimes H_2$ associated to $f$,
$$
K =((H_1\overline\otimes H_2)\otimes\wedge (L_1\times L_2), d^k(f)),
$$
then an easy calculation shows that,
$$
\psi d^k(f) = (d^1(f_1)\otimes 1 + \xi\otimes d^2(f_2))\psi,
$$
where $\xi$ is the map,
$$
\xi\colon\oplus_{p=0}^n(H_1\otimes\wedge^p L_1)\to \oplus_{p=0}^n (H_1\otimes \wedge^p L_1),
$$
$$
\xi = \oplus_{p=0}^n (-1)^p,
$$
 1 is the identity map of $H_1$ and $d^k(f)$ is the boundary map of the Koszul complex $K$, 
equivalently, if we consider the algebraic tensor product of the complexes
$K_1$ and $K_2$, $K_1\otimes K_2$, and its natural completion $K_1\overline\otimes K_2$,
the map $\psi $ provides an isometric isomorphism of chain complexes, from $K$ onto
$K_1\overline\otimes K_2$. \par
\indent Moreover, if $T_i$, $i=1,2$, and $T_k$ are the maps,
$$
T_i= d^{i}(f_i)+{d^{i}(f_i)}^*, \hskip3cm T_k=d^k(f)+ {d^k(f)}^*,
$$
 as $\xi$ is a selfadjoint map, an easy calculation shows that,
$$
\psi T_k= (T_1\otimes 1+\xi\otimes T_2)\psi.
$$
\indent Let us return to the proof. If $f$ does not belong to $Sp(L_1\times L_2, H_1\overline\otimes H_2)$,
by [7, Lemma 3.1], the operator $T_k$ is an invertible map. On the other hand,
if $f_i$, $i=1,2$, belongs to $Sp(L_i,H_i)$, $i=1,2$, as $T_i$ is a selfadjoint operator, i= 1,2,
there exist by [7, Lemma 3.1] and by [4, Chapter II, Section 31, Theorem 2] two sequences of unit vectors, $(a_n^{i})_{n\in \Bbb N}$,
$i=1,2$, such that $\parallel a_n^{i}\parallel =1$,
$a_n^{i}\in H_i\otimes\wedge L_i$, and
$T_i(a_n^{i})\to 0$ ($n\to\infty$), for $i =1,2$. However, as $\psi$ and $\psi^{-1}$ are isometric isomorphisms, by
 elementary properties of the tensor product in Hilbert spaces, we have that: $\parallel a_n^1\otimes a_n^2\parallel=1$,
$\parallel\psi^{-1}(a_n^1\otimes a_n^2)\parallel=1$ and $T_k(\psi^{-1}(a_n^{1}\otimes a_n^2))\to 0$ ($n\to\infty$),
equivalently, by [4, Chapter II, Section 31, Theorem 2], $0\in  Sp(T_k)$, which is impossible by our assumption.\par  
\end{proof}
\bibliographystyle{amsplain}

\vskip.5cm

\noindent Enrico Boasso\par
\noindent E-mail address: enrico\_odisseo@yahoo.it

\end{document}